\documentclass[12pt]{amsart}
\usepackage{amsfonts}
\usepackage{amscd}
\usepackage{amssymb}
\usepackage{graphics}

\usepackage{amsmath}
\usepackage{graphicx,psfrag}
\usepackage{subcaption,mathtools}

\usepackage{hyperref}
\hypersetup{colorlinks,citecolor=blue}

\newcommand{\la}{\langle}
\newcommand{\ra}{\rangle}
\newtheorem{theorem}{Theorem}

\newtheorem{corollary}[theorem]{Corollary}

\newtheorem{lemma}[theorem]{Lemma}

\newtheorem{proposition}[theorem]{Proposition}

\makeatletter
\newtheorem*{rep@theorem}{\rep@title}

\newcommand{\app}{{\sim_H}}
\newcommand{\appK}{{\sim_{K_{n}}}}
\newcommand{\appH}{{\sim_{H_{n}}}}

\newcommand{\Cl}{\mathrm{Cl}}

\newcommand{\zz}{\mathbb{Z}[\frac{1}{2}]}

\newcommand{\modd}{\mathrm{mod}\ }

\theoremstyle{definition}

\catcode`\@=11
\long\def\@savemarbox#1#2{\global\setbox#1\vtop{\hsize\marginparwidth 
  \@parboxrestore\tiny\raggedright #2}}
\catcode`\@=12

\begin{document}
\author{Gili Golan Polak and Mark Sapir}
\thanks{The research of the first author was supported by ISF grant 2322/19. The research of the second author was supported by NSF grant DMS-1901976.}

\title{On some generating set of Thompson's group $F$}

\begin{abstract}
We prove that Thompson's group $F$ has a generating set with two elements such that
every two powers of them generate a finite index subgroup of $F$. 
\end{abstract}

\maketitle

\section{Introduction}

Recall that Thompson's group $F$ is the group of all piecewise linear homeomorphisms of the interval $[0,1]$ where all breakpoints are  dyadic fractions and all slopes are integer powers of $2$.

Thompson's group $F$ has many interesting properties. 
It is infinite and finitely presented,  
  it does not have any free subgroups and it does not satisfy any law \cite{BriS}. 
In 1984, Brown and Geogheghan \cite{BG} proved that Thompson's group $F$ is of type $FP_\infty$, making Thompson's group $F$ the first example of a torsion-free infinite-dimensional $FP_\infty$ group. 

One of the most interesting
and counter-intuitive results about Thompson's group $F$ is that in a certain natural probabilistic model on the set of all finitely generated subgroups of F, every finitely generated nontrivial
subgroup appears with positive probability \cite{CERT}. In \cite{G-RG}, the first author proved  that in the natural probabilistic models studied in \cite{CERT}, a random pair of elements of $F$ generates $F$ with positive probability.
In fact, one can prove that for every finite index subgroup $H$ of $F$, a random pair of elements of $F$ generates $H$ with positive probability. This result  shows that in some sense it is  ``easy'' to generate $F$, or more generally, finite index subgroups of $F$. Several other results in the literature can be interpreted in a similar way. In \cite{G-32}, the first author proved that every element of $F$ whose image in the abelianization $\mathbb{Z}^2$ is part of a generating pair of $\mathbb{Z}^2$ is part of a generating pair of $F$ (and that a similar statement holds for all finitely generated subgroups of $F$).  

Another result that demonstrates the abundance of generating pairs of $F$ is Brin's result \cite{B} that
the free group of rank $2$ is a limit of $2$-markings of Thompson's group $F$ in the space of all $2$-marked groups. Lodha's new (and much shorter) proof \cite{L} of Brin's theorem  demonstrates even better the abundance of generating pairs of $F$. 

In \cite{GGJ}, Gelander, Juschenko and the first author proved that Thompson's group $F$ is invariably generated. 
Recall that a subset $S$ of a group $G$ {\it invariably generates} $G$ if $G= \langle s^{g(s)} | s \in S\rangle$ for every choice of $g(s) \in G,s \in S$. A group $G$ is said to be {\it invariably generated} if such  $S$ exists, or equivalently if $S=G$ invariably generates $G$. Note that all virtually solvable groups are invariably generated, but Thompson's group $F$ was one of the first examples of a  non-virtually solvable group that is invariably generated. 
Note also that in \cite{GGJ} it is proved  that Thompson's group $F$ is invariably generated by a set of $3$ elements. Using  \cite[Theorem 1.3]{G22}, the proof from \cite{GGJ} implies that in fact, Thompson's group $F$ is invariably generated by a set of $2$ elements (see also Lemma \ref{lem:ig} below).

\vskip .2cm

In this paper, we 
prove  a somewhat similar result. 

\begin{theorem}\label{Thm1}
	Thompson group $F$ has a $2$-generating set $\{x,y\}$ such that for every $m,n\in\mathbb{N}$, the set  $\{x^m,y^n\}$ generates a finite index subgroup of $F$.
\end{theorem}

We will show that the generating set $\{x,y\}$ constructed in the proof of  Theorem \ref{Thm1} below
also invariably generates $F$.
Note also that since the abelianization of Thompson's group $F$ is $\mathbb{Z}^2$, we couldn't request the elements $x^m$ and $y^n$ from the theorem to generate the entire group $F$.


Theorem \ref{Thm1} does not hold for any non-elementary hyperbolic group. 
Indeed, if $G$ is non-elementary hyperbolic, then there exists $n\in\mathbb{N}$ such that $G/G^n$ is infinite, where $G^n$ is the normal subgroup generated by all $n^{th}$ powers of elements in $G$ \cite{IO}. More generally, Theorem \ref{Thm1} does not hold for any group $G$ which has an infinite  periodic quotient (such as large groups (see \cite{OO}) and Golod Shafarevich-groups (see \cite{W})).

Theorem \ref{Thm1} does hold for the Tarski monsters constructed by Ol'shanskii \cite{O}.  Recall that Tarski monsters are infinite finitely generated simple groups where every proper subgroup is infinite cyclic\footnote{There is another type of Tarski monsters, where every proper subgroup  is cyclic of order $p$ for some fixed prime $p$, but for them Theorem \ref{Thm1} clearly does not hold.}. Let $T$ be the Tarski monster constructed in \cite{O}, then $2$ elements of $T$ generate it if and only if they do not commute. Since powers of non-commuting elements in $T$ do not commute (see \cite[Theorem 28.3]{O}), any generating pair of $T$ satisfies the assertion in Theorem \ref{Thm1} (in fact, for every pair of generators of $T$, any pair of powers of the generators generates the entire group $T$). It is easy to see that there are virtually-abelian groups (such as $\mathbb{Z}^2$ and $\mathbb{Z}\wr \mathbb{Z}_2$) for which Theorem \ref{Thm1} holds. But to our knowledge, Thompson's group $F$ is the first example of a finitely presented non virtually-abelian group which satisfies the assertion in Theorem \ref{Thm1}. 

\section{Thompson's group F}\label{s:FT}

\subsection{F as a group of homeomorphisms}\label{sec:Fhom}

Recall that Thompson group $F$ is the group of all piecewise linear homeomorphisms of the interval $[0,1]$ with finitely many breakpoints where all breakpoints are  dyadic fractions and all slopes are integer powers of $2$.  
The group $F$ is generated by two functions $x_0$ and $x_1$ defined as follows \cite{CFP}.
	
	\[
   x_0(t) =
  \begin{cases}
   2t &  \hbox{ if }  0\le t\le \frac{1}{4} \\
   t+\frac14       & \hbox{ if } \frac14\le t\le \frac12 \\
   \frac{t}{2}+\frac12       & \hbox{ if } \frac12\le t\le 1
  \end{cases} 	\qquad	
   x_1(t) =
  \begin{cases}
   t &  \hbox{ if } 0\le t\le \frac12 \\
   2t-\frac12       & \hbox{ if } \frac12\le t\le \frac{5}{8} \\
   t+\frac18       & \hbox{ if } \frac{5}{8}\le t\le \frac34 \\
   \frac{t}{2}+\frac12       & \hbox{ if } \frac34\le t\le 1 	
  \end{cases}
\]

The composition in $F$ is from left to right.

Every element of $F$ is completely determined by how it acts on the set $\zz$. Every number in $(0,1)$ can be described as $.s$ where $s$ is an infinite word in $\{0,1\}$. For each element $g\in F$ there exists a finite collection of pairs of (finite) words $(u_i,v_i)$ in the alphabet $\{0,1\}$ such that every infinite word in $\{0,1\}$ starts with exactly one of the $u_i$'s and such that the action of $g$ on a number $.s$ is the following: if $s$ starts with $u_i$, we replace $u_i$ by $v_i$. For example, $x_0$ and $x_1$  are the following functions:

\[
   x_0(t) =
  \begin{cases}
   .0\alpha &  \hbox{ if }  t=.00\alpha \\
    .10\alpha       & \hbox{ if } t=.01\alpha\\
   .11\alpha       & \hbox{ if } t=.1\alpha\
  \end{cases} 	\qquad	
   x_1(t) =
  \begin{cases}
   .0\alpha &  \hbox{ if } t=.0\alpha\\
   .10\alpha  &   \hbox{ if } t=.100\alpha\\
   .110\alpha            &  \hbox{ if } t=.101\alpha\\
   .111\alpha                      & \hbox{ if } t=.11\alpha\
  \end{cases}
\]
where $\alpha$ is any infinite binary word.

The group $F$ has the following finite presentation \cite{CFP}.
$$F=\la x_0,x_1\mid [x_0x_1^{-1},x_1^{x_0}]=1,[x_0x_1^{-1},x_1^{x_0^2}]=1\ra,$$ where $a^b$ denotes $b^{-1} ab$. Sometimes, it is more convenient to consider an infinite presentation of $F$. For $i\ge 1$, let $x_{i+1}=x_0^{-i}x_1x_0^i$. In these generators, the group $F$ has the following presentation \cite{CFP}
$$\la x_i, i\ge 0\mid x_i^{x_j}=x_{i+1} \hbox{ for every}\ j<i\ra.$$

\subsection{Elements of F as pairs of binary trees} \label{sec:tree}

Often, it is more convenient to describe elements of $F$ using pairs of finite binary trees (see \cite{CFP} for a detailed exposition). The considered binary trees are rooted \emph{full} binary trees; that is, each vertex is either a leaf or has two outgoing edges: a left edge and a right edge. A  \emph{branch} in a binary tree is a simple path from the root to a leaf. If every left edge in the tree is labeled ``0'' and every right edge is labeled ``1'', then a branch in $T$ has a natural binary label. We rarely distinguish between a branch and its label. 

Let $(T_+,T_-)$ be a pair of finite binary trees with the same number of leaves. The pair $(T_+,T_-)$ is called a \emph{tree-diagram}. Let $u_1,\dots,u_n$ be the (labels of) branches in $T_+$, listed from left to right. Let $v_1,\dots,v_n$ be the (labels of) branches in $T_-$, listed from left to right. For each  $i=1,\dots,n$, we say that the tree-diagram $(T_+,T_-)$ has the \emph{pair of branches} $u_i\rightarrow v_i$. We also say that the tree-diagram $(T_+,T_-)$ \textit{consists} of all the pairs of branches $u_1\to v_1,\dots, u_n\to v_n$. 
The tree-diagram $(T_+,T_-)$ \emph{represents} the function $g\in F$ which takes binary fraction $.u_i\alpha$ to $.v_i\alpha$ for every $i$ and every infinite binary word $\alpha$. We also say that the element $g$ takes the branch $u_i$ to the branch $v_i$.
For a finite binary word $u$, we denote by $[u]$  the dyadic interval $[.u,.u1^{\mathbb{N}}]$. If $u\rightarrow v$ is a pair of branches of $(T_+,T_-)$, then $g$ maps the interval $[u]$ linearly onto $[v]$. 

A \emph{caret} is a binary tree composed of a root with two children. If $(T_+,T_-)$ is a tree-diagram and one attaches a caret to the $i^{th}$ leaf of $T_+$ and the $i^{th}$ leaf of $T_-$ then the resulting tree diagram is \emph{equivalent} to $(T_+,T_-)$ and represents the same function in $F$. The opposite operation is that of \emph{reducing} common carets. A tree diagram $(T_+,T_-)$ is called \emph{reduced} if it has no common carets; i.e, if there is no $i$ for which the  $i$ and ${i+1}$ leaves of both $T_+$ and $T_-$ have a common father. Every tree-diagram is equivalent to a unique reduced tree-diagram. Thus elements of $F$ can be represented uniquely by reduced tree-diagrams \cite{CFP}.
The reduced tree-diagrams of the generators $x_0$ and $x_1$ of $F$ are depicted in Figure \ref{fig:x0x1}.

\begin{figure}[ht]
	\centering
	\begin{subfigure}{.5\textwidth}
		\centering
		\includegraphics[width=.5\linewidth]{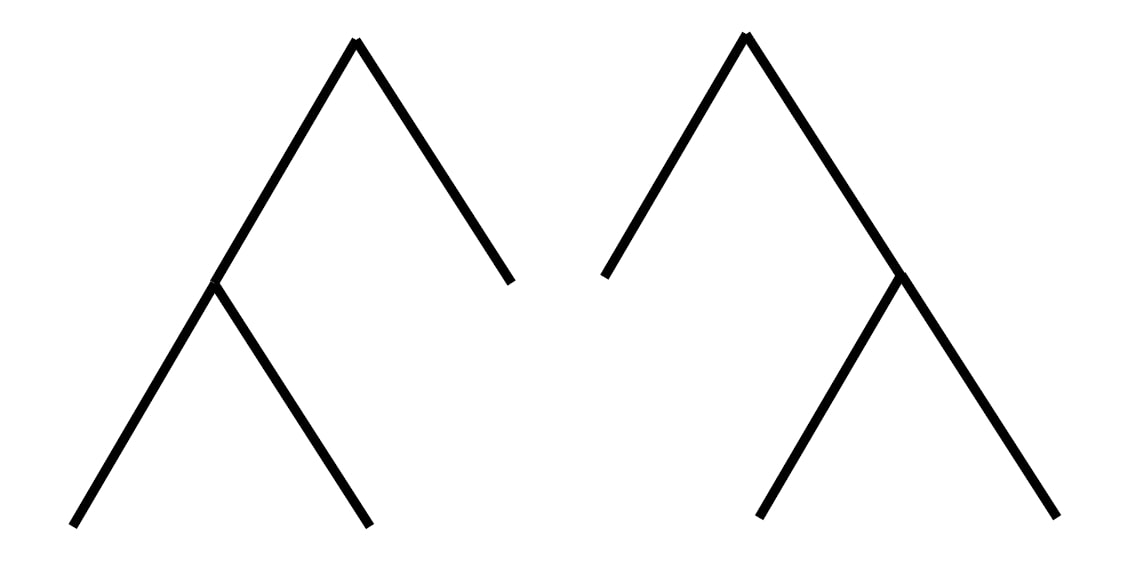}
		\caption{}
		\label{fig:x0}
	\end{subfigure}%
	\begin{subfigure}{.5\textwidth}
		\centering
		\includegraphics[width=.5\linewidth]{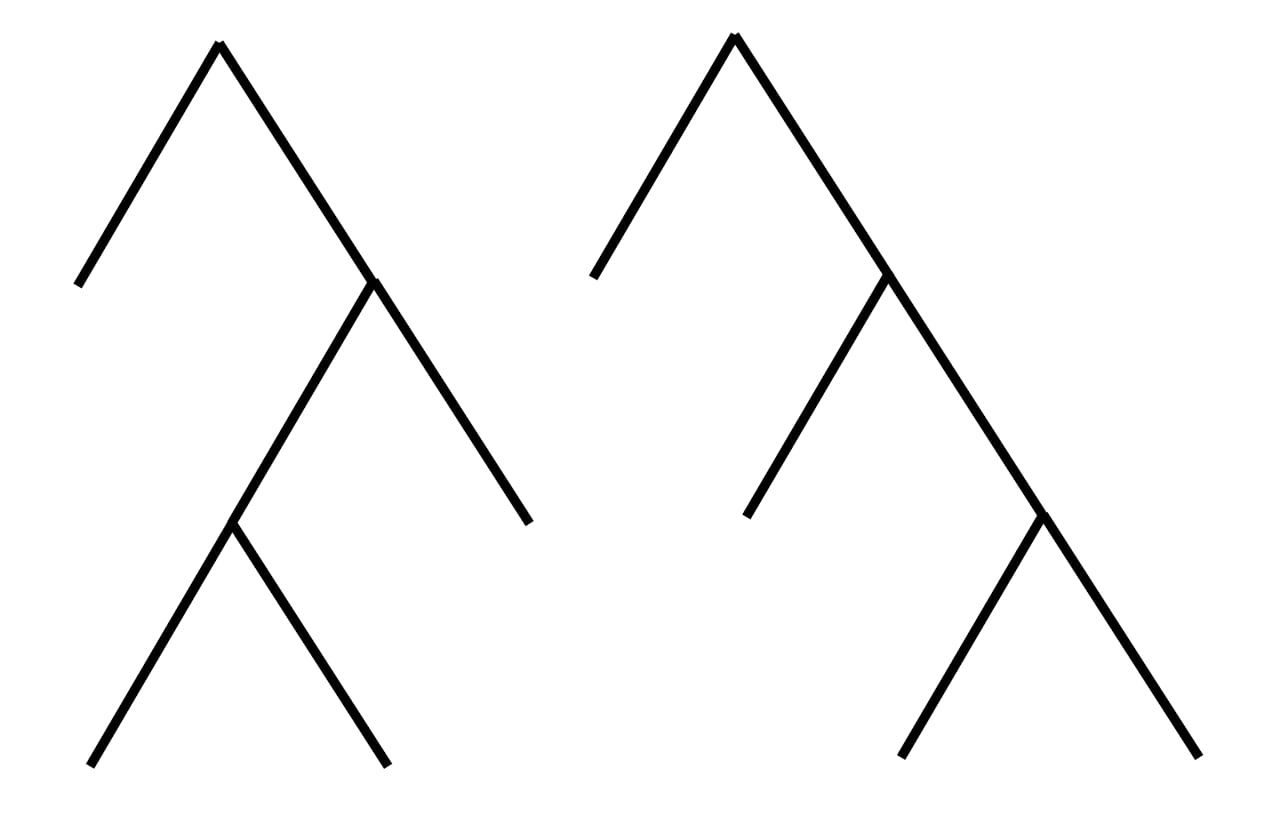}
		\caption{}
		\label{fig:x1}
	\end{subfigure}
	\caption{(A) The reduced tree-diagram of $x_0$. (B) The reduced tree-diagram of $x_1$. In both figures, $T_+$ is on the left and $T_-$ is on the right.}
	\label{fig:x0x1}
\end{figure}

When we say that a function $f\in F$ has a pair of branches $u_i\rightarrow v_i$, the meaning is that some tree-diagram representing $f$ has this pair of branches. In other words, this is equivalent to saying that $f$ maps the dyadic interval $[u_i]$ linearly onto $[v_i]$.
 Clearly, if $u\rightarrow v$ is a pair of branches of $f$, then for any finite binary word $w$, $uw\rightarrow vw$ is also a pair of branches of $f$. Similarly, if $f$ has the pair of branches $u\rightarrow v$ and $g$ has the pair of branches $v\rightarrow w$ then $fg$ has the pair of branches $u\rightarrow w$. 
 
\subsection{The derived subgroup of $F$}\label{sec:derived}

The derived subgroup of $F$ is an infinite simple group \cite{CFP}. It can be characterized as the subgroup of $F$ of all functions $f$ with slope $1$ both at $0^+$ and at $1^-$ (see \cite{CFP}). 
That is, a function $f\in F$ belongs to $[F,F]$ if and only if the reduced tree-diagram of $f$ has pairs of branches of the form $0^m\rightarrow 0^m$ and $1^n\rightarrow 1^n$ for some $m,n\in\mathbb{N}$.

Since $[F,F]$ is infinite and simple, every finite index subgroup of $F$ contains the derived subgroup of $F$.
Hence, there is a one-to-one correspondence between finite index subgroups of $F$ and finite index subgroups of the abelianization $F/[F,F]$. 

Recall that the abelianization of $F$ is isomorphic to $\mathbb{Z}^2$ and that the standard abelianization map $\pi_{ab}\colon F\to\mathbb{Z}^2$ maps an element $f\in F$ to $(\log_2(f'(0^+)),\log_2(f'(1^-)))$. Hence, a subgroup $H$ of $F$ has finite index in $F$ if and only if $H$ contains the derived subgroup of $F$ and $\pi_{ab}(H)$ has finite index in $\mathbb{Z}^2$. 


\subsection{Generating sets of F}

Let $H$ be a subgroup of $F$. A function $f\in F$ is said to be a \emph{piecewise-$H$} function if there is a finite subdivision of the interval $[0,1]$ such that on each interval in the subdivision, $f$ coincides with some function in $H$. Note that since all breakpoints of elements in $F$ are dyadic fractions, a function $f\in F$ is a piecewise-$H$ function if and only if there is a  dyadic subdivision of the interval $[0,1]$ into finitely many pieces such that on each dyadic interval in the subdivision, $f$ coincides with some function in $H$.

Following \cite{GS,G16}, we define the \emph{closure} of a subgroup $H$ of $F$, denoted $\Cl(H)$, to be the subgroup of $F$ of all piecewise-$H$ functions. A subgroup $H$ of $F$ is \emph{closed} if $H=\Cl(H)$. In \cite{G16} (see also \cite{G22}), the first author proved that the generation problem in $F$ is decidable. That is, there is an algorithm that decides given a finite subset $X$ of $F$ whether it generates the whole $F$. 


\begin{theorem}\cite[Theorem 1.3]{G22}\label{thm:H=F}
	Let $H$ be a subgroup of $F$. Then $H=F$ if and only if the following conditions hold. \begin{enumerate}
		\item[(1)] $\Cl(H)$ contains the derived subgroup of $F$. 
		\item[(2)] $H[F,F]=F$. 
	\end{enumerate}
\end{theorem}

More generally, we have a criterion for when a subgroup $H$ of $F$ contains the derived subgroup of $F$.

\begin{theorem}\cite[Theorem 7.10]{G16}\label{gen}
Let $H$ be a subgroup of $F$. Then $H$ contains the derived subgroup $[F,F]$ if and only if the following conditions hold. 
\begin{enumerate}
\item[(1)] $\Cl(H)$ contains the derived subgroup $[F,F]$. 
\item[(2)] There is an element $h\in H$ and a dyadic fraction $\alpha\in (0,1)$ such that $h$ fixes $\alpha$, $h'(\alpha^-)=1$ and 
$h'(\alpha^+)=2$.
\end{enumerate}

\end{theorem}

Below we  apply Theorem \ref{gen} to prove that a given subset of $F$ generates a finite index subgroup of $F$ (by proving that it contains the derived subgroup of $F$ and considering its image in the abelianization of $F$).
The following two lemmas will be useful in proving that Condition (1) of Theorem \ref{gen} holds for a subgroup $H$ of $F$.

\begin{lemma}\label{inner}
	Let $H$ be a subgroup of $F$. Assume that for every pair of finite binary words $u$ and $v$ which both contain both digits $``0"$ and $``1"$ there is an element $h\in H$ with the pair of branches $u\rightarrow v$. Then $\Cl(H)$ contains the derived subgroup of $F$. 
\end{lemma}

\begin{proof}
		 Let $f\in [F,F]$. Then the reduced tree-diagram of $f$ consists of the pairs of branches 
	\[
	f :
	\begin{cases}
	0^m & \rightarrow 0^m\\
	u_i  & \rightarrow v_i \mbox{ for } i=1,\dots,k \\
	1^n & \rightarrow 1^n\\
	\end{cases}
	\]
	where $k,m,n\in\mathbb{N}$ and where for each $i=1,\dots,k$, the binary words $u_i$ and $v_i$ contain both digits $``0"$ and $``1"$. By assumption, for each $i=1,\dots,k$ there is an element $h_i\in H$ with the pair of branches $u_i\rightarrow v_i$. Then $h_i$ coincides with $f$ on the interval $[u_i]$. we note also that $f$ coincides with the identity function {\bfseries{1}} $\in H$ on $[0^m]$ and on $[1^n]$. Since $[0^m],[u_1],\dots,[u_k],[1^n]$ is a subdivision of the interval $[0,1]$ and on each of these intervals $f$ coincides with a function in $H$, $f$ is a piecewise-$H$ function and as such $f\in \Cl(H)$. 
\end{proof}

Given a subgroup $H\le F$ we associate with $H$ an equivalence relation on the set of finite binary words as follows. 
Let $u$ and $v$ be finite binary words. We write $u\app v$ if there is an element $h\in H$ with the pair of branches $u\rightarrow v$. Note that $\app$ is indeed an equivalence relation on the set of finite binary words. 
(Indeed, for every finite binary word $u$ the identity function has the pair of branches $u\to u$; if $h\in H$ has the pair of branches $u\to v$ then $h^{-1}$ has the pair of branches $v\to u$ and if $h,g\in H$ have the pairs of branches $u\to v$ and $v\to w$, respectively, then $hg$ has the pair of branches $u\to w$). We note also that if $u\app v$ then for any finite binary word $w$ we have $uw\app vw$. Indeed, if $h\in H$ has the pair of branches $u\rightarrow v$ then for each $w$ (some non-reduced tree-diagram of) $h$ has the pair of branches $uw\rightarrow vw$. 
By Lemma \ref{inner}, to prove that $\Cl(H)$ contains the derived subgroup of $F$, it suffices to prove that all finite binary words which contain both digits ``0" and ``1" are $\app$-equivalent. 

\begin{lemma}\label{lem:suffice}
	Let $H$ be a subgroup of $F$ such that the following assertions hold. 
	\begin{enumerate}
		\item[(1)] For every $r\in\mathbb{N}$, we have $1^r0\sim_H 10$. 
		\item[(2)] For every $s\in\mathbb{N}$, we have $0^s1\sim_H 01$. 
		\item[(3)] $01\sim_H 10\sim_H 010\sim_H 011$. 
	\end{enumerate}
Then $\Cl(H)$ contains the derived subgroup of $F$. 
\end{lemma}

\begin{proof}
	First, note that since $10\sim_H 01$, we have $100\sim_H 010$ and $101\sim_H 011$. Then $(3)$ implies that
	$$(4) \ \ \ 100
	\sim_H 010 \sim_H 01\sim_H 011\sim_H 101.$$ 
	Now, let $u$ be a finite binary word which contains both digits ``0'' and ``1''. It suffices to prove that  $u\sim_H 01$ (indeed, in that case, all finite binary words which contain both digits ``0'' and ``1'' are $\sim_H$-equivalent). If $u$ is of length $2$, this is true, since $10\sim_H 01$. If $u$ is of length $\geq 3$, then it must have a prefix of the form  $1^r0$ (for some $r\ge 2$), $0^s1$ (for some $s\ge 2$), $010$, $011$, $100$ or $101$. In all of these cases, $u$ is $\sim_H$-equivalent to a shorter word (since it has a prefix that is $\sim_H$-equivalent to a shorter word by (1)-(4) above). Hence, we are done by induction. 
\end{proof}

\section{Proof of Theorem 1}\label{sec:proof}

For the rest of this section, let $x=x_0$ and $y=x_0^2x_1$ (the element $x$ appears in Figure \ref{fig:x0} and the element $y$ appears in Figure \ref{fig:y}). Since $\{x_0,x_1\}$ is a generating set of $F$, the set $\{x,y\}$ is a generating set of $F$. We will prove that for every $m,n\in\mathbb{N}$ the set $\{x^m,y^n\}$ generates a finite index subgroup of $F$ and that $\{x,y\}$ invariably generates $F$. 

We begin with the following lemma. 


\begin{figure}[ht]
		\centering
		\includegraphics[width=.35\linewidth]{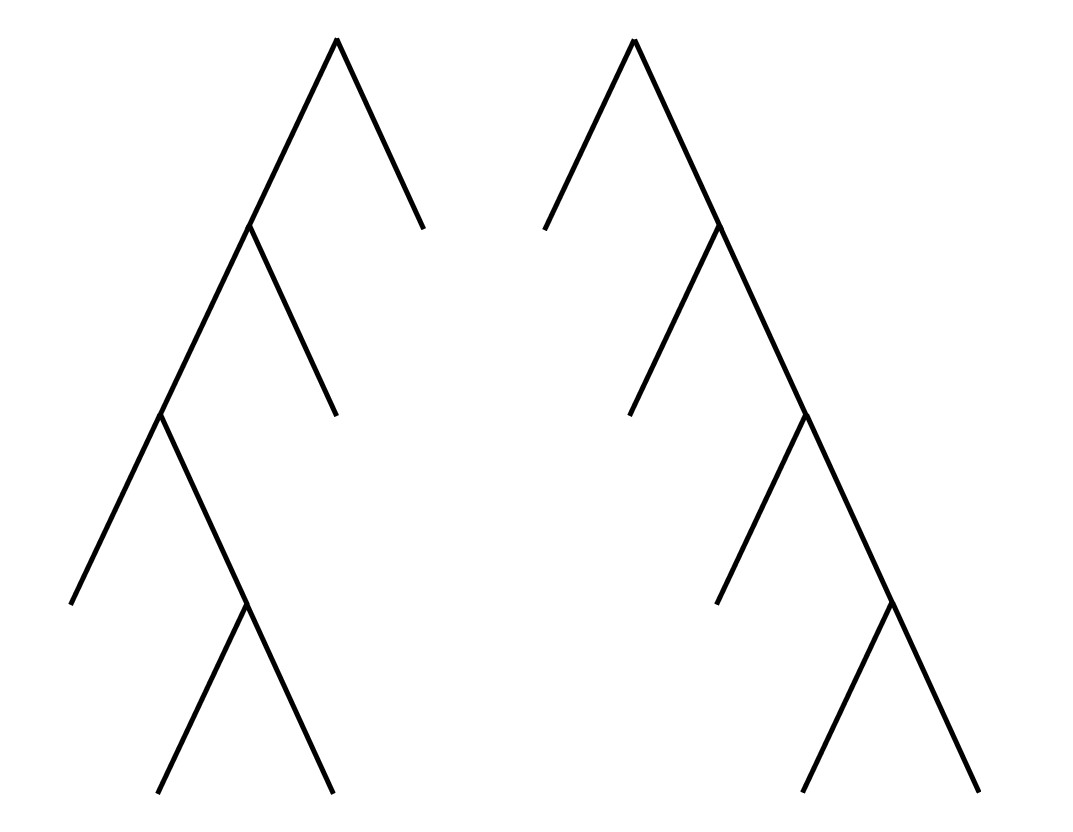}
	\caption{The reduced tree-diagram of $y$.}
	\label{fig:y}
\end{figure}


\begin{lemma}\label{lem:bra}
	Let $n\in\mathbb{N}$. Then the reduced tree diagrams of $x^n$ and $y^n$ consist of the following pairs of branches (that is, we list all the pairs of branches of $x^n$ and $y^n$). 
	\[
	x^n :
	\begin{cases}
	0^{n+1} & \rightarrow 0\\
	0^k1  & \rightarrow 1^{n+1-k}0, \mbox{ for } 1\le k\le n\\
	1 & \rightarrow 1^{n+1}\\
	\end{cases}
	\]
	\[
	y^n :
	\begin{cases}
	0^{2n+1} & \rightarrow 0\\
	0^{2k}10  & \rightarrow 1^{1+3(n-k)}0,  \mbox{ for } 1\le k\le n\\
	0^{2k}11 & \rightarrow 1^{2+3(n-k)}0, \mbox{ for } 1\le k\le n\\
	0^{2k-1}1 & \rightarrow 1^{3(n-k+1)}0,  \mbox{ for } 1\le k\le n\\
	1 & \rightarrow 1^{3n+1}
	\end{cases}
	\]
\end{lemma}

\begin{proof}
	The lemma can be proved by induction. Note that for $n=1$ the lemma follows from Figure \ref{fig:x0} and from Figure \ref{fig:y}.
\end{proof}


Now, for every $n\in\mathbb{N}$, we denote by $H_{n}$ the subgroup of $F$ generated by $\{x^n,y^n\}$. We claim that $H_{n}$ contains the derived subgroup of $F$. To prove that, we will prove that $H_{n}$ satisfies Conditions $(1)$ and $(2)$ from Theorem \ref{gen}. First, we consider Condition $(2)$. 

\begin{lemma}\label{lem:slope}
	Let $n\in\mathbb{N}$. Then there is an element $h\in H_{n}$ such that $h$ fixes a dyadic fraction $\alpha\in (0,1)$ and such that $h'(\alpha^-)=1$ and $h'(\alpha^+)=2$.  
\end{lemma}

\begin{proof}
	From the infinite presentation of $F$ given above it follows  that 
	$$y^n=(x_0^2x_1)^n=x_0^{2n}x_1x_4x_7\cdots x_{1+3(n-1)}.$$ 
	Since $x^{2n}=x_0^{2n}\in H_{n}$ we have that 
	$$h=x_1x_4x_7\cdots x_{1+3(n-1)}\in H_{n}.$$
	Note that for $\alpha=\frac{1}{2}$ the function $x_1$ fixes $[0,\alpha]$ pointwise and satisfies $x_1'(\alpha^+)=2$.
	For all $i>1$, the function $x_i$ fixes $[0,\frac{3}{4}]$ pointwise, hence for $\alpha=\frac{1}{2}$ we have $h(\alpha)=\alpha$, $h'(\alpha^-)=1$ and $h'(\alpha^+)=2$. 
\end{proof}

 To prove that Condition $(1)$ from Theorem \ref{gen} holds for $H_{n}$, we let
$K_{n}$ be the minimal closed subgroup of $F$ 
 such that the following hold modulo $\appK$. 
\begin{equation*}
\begin{aligned}
&(a)\ \  0^{k}1\appK 0^{k+n}1, &&& \mbox{ for all } k\in\mathbb{N}\\
&(b)\ \ 1^k0 \appK 1^{k+n}0, &&& \mbox{ for all } k\in\mathbb{N}\\
&(c)\ \ 0^k1\appK 1^{n+1-k}0, &&& \mbox{ for $ 1\le k\le n$}\\
&(d)\ \ 0^{2k}10 \appK 1^{1+3(n-k)}0, &&&  \mbox{ for } 1\le k\le n\\
&(e)\ \ 0^{2k}11\appK 1^{2+3(n-k)}0, &&& \mbox{ for } 1\le k\le n\\
&(f)\ \ 0^{2k-1}1 \appK 1^{3(n-k+1)}0, &&& \mbox{ for } 1\le k\le n.\\
\end{aligned}
\end{equation*}


Note that the intersection of closed subgroups of $F$ is a closed subgroup (see \cite{G16}) and that modulo $\sim_F$ relations $(a)-(f)$ hold. Hence, $K_{n}$ is well defined. 

\begin{lemma}\label{lem1}
	Let $n\in\mathbb{N}$. 
	Then $K_{n}\subseteq \Cl(H_{n})$.
\end{lemma}

\begin{proof}
	It suffices to prove that equivalences $(a)-(f)$ hold when $K_{n}$ is replaced by $H_{n}$. 
	Indeed, in that case, the equivalences must also hold modulo $\sim_{\Cl(H_{n})}$ 
	and then the minimality of $K_{n}$ implies that it is a subgroup of $\Cl(H_{n})$. 
	
	Let us consider the relation $\appH$. 
	Equivalences $(d),(e),(f)$ are true modulo $\appH$ since $y^n\in H_{n}$. Similarly, $(c)$ holds modulo $\appH$ since $x^n\in H_{n}$. 
	The branch $0^{n+1}\rightarrow 0$ of $x^n$ implies that for all $k\in\mathbb{N}$, 	$0^k\appH 0^{k+n}$. In particular, for all $k\in\mathbb{N}$, we have	$0^k1\appH 0^{k+n}1$, so $(a)$ also holds modulo $\appH$. Finally, the branch $1\to 1^{n+1}$ of $x^n$ implies that for all $k\in\mathbb{N}$, 	$1^k\appH 1^{k+n}$. Hence, for all $k\in\mathbb{N}$, we have $1^k0\appH 1^{k+n}0$, so $(b)$ also holds modulo $\appH$.
\end{proof}

By Lemma \ref{lem1}, to prove that $[F,F]$ is contained in the closure of $H_{n}$ for every $n\in\mathbb{N}$, it suffices to prove that $[F,F]\subseteq K_{n}$ for every $n\in\mathbb{N}$. To do so, we will make use of the following lemma.

\begin{lemma}\label{lem2}
	Let $n\in\mathbb{N}$. If $2|n$ then $K_{\frac{n}{2}}\subseteq K_{n}$. If $3|n$ then $K_{\frac{n}{3}}\subseteq K_{n}$. 
\end{lemma}

\begin{proof}
	Assume that $2|n$. The proof for the case $3|n$ is similar. $K_{\frac{n}{2}}$ is the minimal closed subgroup such that 
			\begin{equation*}
	\begin{aligned}
	&(a')\ \  0^{k}1\sim_{K_{\frac{n}{2}}} 0^{k+\frac{n}{2}}1, &&& \mbox{ for all } k\in\mathbb{N}\\
	&(b')\ \ 1^k0 \sim_{K_{\frac{n}{2}}} 1^{k+\frac{n}{2}}0, &&& \mbox{ for all } k\in\mathbb{N}\\
	&(c')\ \ 0^k1\sim_{K_{\frac{n}{2}}} 1^{\frac{n}{2}+1-k}0, &&& \mbox{ for } 1\le k\le \frac{n}{2}\\
	&(d')\ \ 0^{2k}10\sim_{K_{\frac{n}{2}}} 1^{1+3(\frac{n}{2}-k)}0, &&&  \mbox{ for } 1\le k\le \frac{n}{2}\\
	&(e')\ \ 0^{2k}11\sim_{K_{\frac{n}{2}}} 1^{2+3(\frac{n}{2}-k)}0, &&& \mbox{ for } 1\le k\le \frac{n}{2}\\
	&(f')\ \ 0^{2k-1}1\sim_{K_{\frac{n}{2}}} 1^{3(\frac{n}{2}-k+1)}0, &&& \mbox{ for } 1\le k\le \frac{n}{2}.\\
	\end{aligned}
	\end{equation*}
	It suffices to prove that $(a')-(f')$ hold with $K_{\frac{n}{2}}$ replaced by $K_{n}$.
	We would make use of equivalences $(a)-(f)$ above holding modulo $\sim_{K_{n}}$.
		 
	For every $k=1,\dots,\frac{n}{2}$ we have by $(a)$ and $(d)$ that
	\begin{equation}\label{eq12}
	\begin{aligned}
	0^{2k}10\sim_{K_{n}}0^{2k+n}10= 0^{2(k+\frac{n}{2})}10\sim_{K_{n}}
	1^{1+3(n-k-\frac{n}{2})}0= 1^{1+3(\frac{n}{2}-k)}0.
	\end{aligned}
	\end{equation}
	Hence $(d')$ holds for $K_{n}$. Similarly, by $(a)$ and $(e)$, for every $k=1,\dots,\frac{n}{2}$ we have
		\begin{equation}\label{eq13}
	\begin{aligned}
0^{2k}11\sim_{K_{n}}0^{2k+n}11=0^{2(k+\frac{n}{2})}11\sim_{K_{n}}1^{2+3(n-k-\frac{n}{2})}0=1^{2+3(\frac{n}{2}-k)}0.
	\end{aligned}
	\end{equation}
	Hence, $(e')$ holds modulo $\sim_{K_{n}}$.
	 Similarly, by $(a)$ and $(f)$, for every $k=1,\dots,\frac{n}{2}$ we have
	\begin{equation}\label{eq14}
	\begin{aligned}
	0^{2k-1}1\sim_{K_{n}}0^{2k+n-1}1=0^{2(k+\frac{n}{2})-1}1\sim_{K_{n}}1^{3(\frac{n}{2}-k+1)}0,
	\end{aligned}
	\end{equation}
	so $(f')$ also holds with $K_{\frac{n}{2}}$ replaced by  $K_{n}$.
	
	To finish, it suffices to prove that equivalences $(a'),(b')$ and $(c')$ hold modulo $\sim_{K_{n}}$.
	Since $(b)$ holds modulo $\sim_{K_{n}}$, to prove $(b')$, it suffices to prove that for all $k\in\{1,\dots,\frac{n}{2}\}$ we have 
	$1^k0\sim_{K_{n}}1^{k+\frac{n}{2}}0$. So let $k\in\{1,\dots,\frac{n}{2}\}$ and let $i\in\{1,2,3\}$ be such that $i\equiv k \pmod{3}$. Let $r=\frac{n}{2}-\frac{k-i}{3}$ and note that $r\in \{1,\dots,\frac{n}{2}\}$. We will assume that $i=3$, the proof for $i=1,2$ is similar. Note that if $i=3$ then $k=3(\frac{n}{2}-r+1)$. Then, by $(\ref{eq14}),(f)$ and $(b)$, we have
	\begin{equation}\label{eq15}
	\begin{aligned}
1^k0 &= 1^{3(\frac{n}{2}-r+1)}0\sim_{K_{n}}
0^{2r-1}1
\sim_{K_{n}}1^{3(n-r+1)}0\\
&\sim_{K_{n}}
1^{3+3n-3r-n}0= 1^{3+2n-3r}0= 1^{3(\frac{n}{2}-r+1)+\frac{n}{2}}0= 1^{k+\frac{n}{2}}0.
	\end{aligned}
	\end{equation}
	Thus $(b')$ holds for $K_{n}$.
	To prove that $(a')$ holds for $K_{n}$ we note that for all $k=1,\dots,\frac{n}{2}$, by applying $(c)$ followed by  $(b')$ for $\sim_{K_{n}}$ followed by $(c)$ again, we have
 \begin{equation}\label{eq16}
 \begin{aligned}
 0^k1\sim_{K_{n}}1^{n+1-k}0\sim_{K_{n}}
 1^{n+1-k-\frac{n}{2}}0=
 1^{n+1-(k+\frac{n}{2})}0 
 \sim_{K_{n}}0^{k+\frac{n}{2}}1.
 \end{aligned}
 \end{equation}
 Since $(a)$ holds for $K_{n}$, $(\ref{eq16})$ implies that $(a')$ holds for $K_{n}$ as well. 

 Finally, $(\ref{eq16})$ shows that for all $k\in\{1,\dots\frac{n}{2}\}$ we have 
  \begin{equation}\label{eq17}
 \begin{aligned}
 0^k1\sim_{K_{n}}
 1^{n+1-(k+\frac{n}{2})}0 
 =1^{\frac{n}{2}+1-k}0.
 \end{aligned}
 \end{equation}
 Hence, $(c')$ also holds for $K_{n}$. 
\end{proof}


\begin{proposition}\label{prop}
	Let $n\in\mathbb{N}$. Then $K_{n}$ contains the derived subgroup of $F$. 
\end{proposition}

\begin{proof}
	We prove the proposition by induction on $n$. 
	If $n$ is divisible by $2$ or $3$, then by Lemma \ref{lem2}, we are done by induction. 
	Hence, we can assume that $n$ is not divisible by $2$ nor by $3$. By Lemma \ref{lem:suffice}, to prove that the closed subgroup $K_{n}$ contains the derived subgroup of $F$, it suffices to prove that Conditions (1)-(3) of Lemma \ref{lem:suffice} hold for $K_{n}$. 
	
	By $(a)$ and $(c)$ we have 
	\begin{equation}\label{eq18}
	\begin{aligned}
	0^{2n}10\sim_{K_{n}}0^n10\sim_{K_{n}}1^{n+1-n}00=100. 
	\end{aligned}
	\end{equation}
	On the other hand, by $(d)$ we have 
	\begin{equation}\label{eq19}
	\begin{aligned}
	0^{2n}10\sim_{K_{n}}1^{1+3(n-n)}0=10.
	\end{aligned}
	\end{equation}
	Hence, 
	\begin{equation}\label{eq20}
	100\sim_{K_{n}}10. 
	\end{equation}

	Similarly,	by $(a)$ and $(c)$ we have 
	\begin{equation}\label{eq21}
	\begin{aligned}
	0^{2n}11\sim_{K_{n}} 0^n11\sim_{K_{n}} 10^{n-n+1}1= 101.
	\end{aligned}
	\end{equation}
	By $(e)$ we have 
	\begin{equation}\label{eq22}
	\begin{aligned}
	0^{2n}11\sim_{K_{n}} 1^{2+3(n-n)}0= 110.
	\end{aligned}
	\end{equation}
	Hence,
	\begin{equation}\label{eq23}
	110\sim_{K_{n}}101.
	\end{equation}

	Now, we make the observation that if Condition (1) of Lemma \ref{lem:suffice} holds for $K_{n}$, then Conditions (2) and (3) of Lemma \ref{lem:suffice} also  hold for $K_{n}$. Indeed, assume that for all $r\in\mathbb{N}$ we have 
	 $1^{r}0\sim_{K_{n}}10$.  Then in particular, $110\sim_{K_{n}}10$. Then, it follows from (\ref{eq23}) and (\ref{eq20}) that for all $r\in\mathbb{N}$,
	\begin{equation}\label{eq24}
	\begin{aligned}
	1^r0\sim_{K_{n}}101\sim_{K_{n}}100\sim_{K_{n}}10.
	\end{aligned}
	\end{equation}
	In addition, $(a)$ and $(c)$ from the definition of $K_{n}$ show that for every $s\in\mathbb{N}$ there is some $r\in\mathbb{N}$ such that $0^s1\sim_{K_{n}} 1^r0$. 
	Then it follows from $(\ref{eq24})$ that for all $s\in\mathbb{N}$, $0^s1\sim_{K_{n}} 10$.
	In particular, $01\sim_{K_{n}} 10$. Hence, $0^s1\sim_{K_{n}} 01$ for all $s\in\mathbb{N}$,	
	  so $K_{n}$ satisfies Condition (2) of Lemma \ref{lem:suffice}.  In addition, since $01\sim_{K_{n}} 10$,  we have 
	 $010\sim_{K_{n}} 100\sim_{K_{n}}10$ and $011\sim_{K_{n}} 101\sim_{K_{n}} 10$. 
	Hence, 
	\begin{equation}\label{eq25}
	\begin{aligned}
	010\sim_{K_{n}}011\sim_{K_{n}} 
	10\sim_{K_{n}} 01.
	\end{aligned}
	\end{equation}
	Therefore, $K_{n}$ satisfies Condition (3) of Lemma \ref{lem:suffice} as well.  
	
	Hence, it suffices to prove that Condition (1) of Lemma \ref{lem:suffice} holds for $K_{n}$, i.e., that for every $r\in\mathbb{N}$ we have $1^{r}0\sim_{K_{n}}10$.
	
	 Since $n$ is co-prime to $2$ and $3$ there are $b,c\in\{1,\dots,n\}$ such that $2b\equiv 1 \pmod{n}$ and $3c\equiv 1\pmod{n}$. Below, whenever an integer modulo $n$ appears as an exponent of the digit ``0'' or ``1''   we assume that the chosen representative is in $\{1,\dots,n\}$.
	Recall that by $(a)$ and $(b)$ for $K_{n}$, for all $k\in\mathbb{N}$ we have that $0^{k}1\sim_{K_{n}} 0^{k (\modd{n})}1$ and $1^{k}0\sim_{K_{n}} 1^{k(\modd{n})}0$. We use this fact below, sometimes with no explicit reference.
	
	 We will need the following lemma. 

	 \begin{lemma}\label{lem:ezer}
	 	Let $q\in\mathbb{N}$ be such that $1^q0\sim_{K_{n}} 10$. Then 
	 	$10\sim_{K_{n}} 1^{q-c(\modd n)}0$.
	 \end{lemma}
 
 	\begin{proof}
 		Let $p\in\mathbb{N}$ and let $s\in\{1,\dots,n\}$ be such that $s\equiv 1-bp\pmod{n}$. Then $p\equiv 2-2s\pmod{n}$. Since $s\in\{1,\dots,n\}$,
 		  by $(f)$ followed by $(b)$	we have
 		\begin{equation}\label{eq26}
 		\begin{aligned}
		0^{2s-1}1\sim_{K_{n}}1^{3(n+1-s)}0\sim_{K_{n}}1^{3-3s(\modd{n})}0= 1^{3-3(1-bp)(\modd{n})}0=1^{3bp(\modd{n})}0.
 		\end{aligned}
 		\end{equation}
 		On the other hand, by $(a),(c)$ and $(b)$
 		\begin{equation}\label{eq27}
 		 \begin{aligned}
 0^{2s-1}1\sim_{K_{n}}0^{2s-1(\modd n)}1\sim_{K_{n}}1^{1+n-(2s-1)(\modd n)}0\sim_{K_{n}} 1^{2-2s(\modd n)}0= 1^{p(\modd n)}0.
 	\end{aligned}
 	\end{equation}	 
	Hence, 
	\begin{equation}\label{eq28}
	\begin{aligned}
	1^{p(\modd n)}0\sim_{K_{n}}1^{3bp(\modd n)}0.
	\end{aligned}
	\end{equation} 	
	Since $(\ref{eq28})$ holds for every $p\in\mathbb{N}$ and $(3b)(2c)\equiv 1\pmod{n}$, we have that for all $p\in\mathbb{N}$,
		\begin{equation}\label{eq29}
	\begin{aligned}
	1^{p(\modd n)}0=  1^{3b(2cp)(\modd n)}0
	\sim_{K_{n}}1^{2cp(\modd n)}0.
	\end{aligned}
	\end{equation} 
	
 	Now, let $t\in\{1,\dots,n\}$ be such that $t\equiv b(1-q)\pmod{n}$ and note that $q\equiv 1-2t\pmod{n}$. Then by $(b),(c),(d)$ and the fact that $3b-1\equiv b\pmod{n}$ (indeed, $2b\equiv 1\pmod{n}$), we have
 	\begin{equation}\label{eq30}
 	\begin{aligned}
 	1^q00&\sim_{K_{n}}1^{n+1-2t(\modd n)}00\sim_{K_{n}}0^{2t(\modd n)}10
 	\sim_{K_{n}}1^{1+3(n-t)(\modd n)}0\\
 	&=1^{1-3t(\modd n)}0=
 	1^{1-3b(1-q)(\modd n)}0= 1^{3bq-3b+1(\modd n)}0= 1^{3bq-b(\modd n)}0 .
 	\end{aligned}
 	\end{equation}	
 	Now, since by assumption $1^q0\sim_{K_{n}}10$ and by $(\ref{eq20})$ we have $10\sim_{K_{n}}100$, it follows that $1^q00\sim_{K_{n}} 100\sim_{K_{n}} 10$.  Then from equivalence $(\ref{eq30})$ it follows that
 	\begin{equation}\label{eq31}
 	\begin{aligned}
 	10\sim_{K_{n}}1^{3bq-b(\modd n)}0.
 	\end{aligned}
 	\end{equation} 
 	Then $(\ref{eq31})$ and $(\ref{eq29})$ imply that 
 	\begin{equation}\label{eq32}
 	\begin{aligned}
 	10\sim_{K_{n}}1^{3bq-b(\modd n)}0\sim_{K_{n}}
 	1^{2c(3bq-b)(\modd n)}0\sim_{K_{n}}1^{q-c(\modd n)}0
 	\end{aligned}
 	\end{equation}
 	as required. 
 	\end{proof}
 
Now we can finish proving the proposition. 
 	By lemma \ref{lem:ezer} applied to $q=1$, we get that $10\sim_{K_{n}} 1^{1-c(\modd n)}0$. Another application of the lemma, now for $q\in\mathbb{N}$ such that $q\equiv 1-c\pmod{n}$ shows that $10\sim_{K_{n}}1^{1-2c(\modd n)}0$. Continuing inductively, we get that for all $\ell\in\mathbb{N}$, we have 
 	\begin{equation}\label{eq33}
 	\begin{aligned}
 	10\sim_{K_{n}}1^{1-\ell c(\modd n)}0.
 	\end{aligned}
 	\end{equation}
 	Now, for each $r\in\mathbb{N}$, let $\ell\in\mathbb{N}$ be such that 
 	$\ell \equiv 3(1-r) \pmod{n}$. Then  $r\equiv 1-c\ell \pmod{n}$
 	and by $(\ref{eq31})$ 
 	 we have 
 \begin{equation}\label{eq34}
 \begin{aligned}
 1^r0\sim_{K_{n}} 1^{1-\ell c(\modd n)}0\sim_{K_{n}}10,
 \end{aligned}
 \end{equation}	
 as required. Hence, the proposition holds. 
\end{proof}

\begin{corollary}\label{cor:derived}
	For every $n\in\mathbb{N}$, the subgroup $H_{n}$ contains the derived subgroup of $F$.
\end{corollary}

\begin{proof}
	Let $n\in\mathbb{N}$. Proposition \ref{prop} and Lemma \ref{lem1} imply that the derived subgroup of $F$ is contained in $\Cl(H_{n})$. Hence, Condition (1) of Theorem \ref{gen} holds for $H_{n}$. Lemma \ref{lem:slope} shows that Condition (2) of Theorem \ref{gen} also holds for $H_{n}$. Hence, by Theorem \ref{gen}, $H_{n}$ contains the derived subgroup of $F$. 
\end{proof}


The following lemma completes the proof of Theorem \ref{Thm1}.

\begin{lemma}
	Let $m,n\in\mathbb{N}$. Then $G=\la x^m,y^n\ra$ is a subgroup of $F$ of index $mn$.
\end{lemma}

\begin{proof}
	Note that $x^{mn},y^{mn}\in G$. Hence, $H_{mn}$ is a subgroup of $G$. Hence, by Corollary \ref{cor:derived}, the derived subgroup $[F,F]\le G$. Recall the map $\pi_{ab} \colon F\to \mathbb{Z}^2$ from Section \ref{sec:derived}. 
	By Lemma \ref{lem:bra}, $x^m$ has the pairs of branches $0^{m+1}\rightarrow 0$ and $1\rightarrow 1^{m+1}$. Hence, $\pi_{ab}(x^m)=(m,-m)$. Similarly, 
	$\pi_{ab}(y^n)=(2n,-3n)$. 
	Hence, $\pi_{ab}(G)=\la (m,-m),(2n,-3n)\ra$. Since $\la (m,-m),(2n,-3n)\ra$ is a subgroup of $\mathbb{Z}^2$  of index $|-3mn+2mn|=mn$,  
	the subgroup $G$ is a subgroup  of $F$ of index $mn$, as required. 
\end{proof}

We finish with the following lemma.

\begin{lemma}\label{lem:ig}
	The set $\{x,y\}$ invariably generates $F$. 
\end{lemma}

\begin{proof}
	It suffices to prove that for any $g\in F$, the set $\{x,y^g\}$ is a generating set of $F$. 
	Let $g\in F$ and let $m,n\in\mathbb{Z}$ be such that $\pi_{ab}(g)=(m,n)$. Since $\{\pi_{ab}(x),\pi_{ab}(y)\}$ generates $\mathbb{Z}^2$, there exist $i,j\in\mathbb{Z}$ such that $i\pi_{ab}(x)+j\pi_{ab}(y)=-(m,n)$. Let $h=y^jgx^i$ and note that $h\in [F,F]$.  Now, $\{x, y^g\}$ generates $F$ if and only if so does $\{x^{x^{i}}, y^{gx^{i}}\}=\{x,y^{y^{-j}h}\}=\{x,y^h\}$.

	Let $H$ be the subgroup of $F$ generated by $X=\{x,y^h\}$. Then $H[F,F]=F$ (indeed, the image of $X$ in the abelianization of $F$ coincides with the image of the generating set $\{x,y\}$). Hence, by Theorem \ref{thm:H=F}, to prove that $H=F$ it suffices to prove that $\Cl(H)$ contains the derived subgroup of $F$. For that, we will make use of Lemma \ref{lem:suffice}. 
Since $x=x_0\in H$ has the pairs of branches $00\to 0$, $01\to 10$ and $1\to 11$, we have that for all $k\in\mathbb{N}$, $0^k\sim_H 0$, $1^k\sim_H 1$ and $10\sim_H 01$. In particular, for every $k\in\mathbb{N}$, we have $0^k1\sim_H 01$ and $1^k0\sim_H 10$. Hence, Conditions (1) and (2) of Lemma \ref{lem:suffice} hold for $H$. To prove that Condition (3) from Lemma \ref{lem:suffice} holds as well, it suffices to prove that $010\sim_H 011\sim_H 10$.

Let us consider the element $h$. Since $h\in [F,F]$, there exist $a,b\in\mathbb{N}$ such that $h$ has the pair of branches $0^a\to 0^a$ and $1^b\to 1^b$. Let $n=\max\{a,b\}$ and consider the element $f=h^{-1}y^{2n}h\in H$. We claim that $f$ has the pairs of branches 
\begin{enumerate}
	\item[(1)] $0^{2n}10\to 1^{1+3n}0$,
	\item[(2)] $0^{2n}11\to 1^{2+3n}0$.
\end{enumerate}
Indeed, by Lemma \ref{lem:bra}, the element $y^{2n}$ has the pairs of branches $0^{2n}10\to 1^{1+3(2n-n)}0= 1^{1+3n}0$ and $0^{2n}11\to 1^{2+3(2n-n)}0=1^{2+3n}0$. Since $h$ fixes the intervals $[0^{2n}]\subseteq [0^a]$ and $[1^{3n}]\subseteq [1^b]$ pointwise, the element $f$ also has the pairs of branches $0^{2n}10\to 1^{1+3n}0$ and $0^{2n}11\to 1^{2+3n}0$, as claimed. 



Now, from (1) and the fact that for all $k\in\mathbb{N}$, we have $0^k\sim_H 0$ and $1^k\sim_H 1$, we have that $010\sim_H 10$. Similarly, using (2), we get that $011\sim_H 10$. Hence, Condition (3) of Lemma \ref{lem:suffice} holds for $H$. Since $H$ satisfies all the conditions of Lemma \ref{lem:suffice}, $\Cl(H)$ contains the derived subgroup of $F$, as necessary. 
\end{proof}

\textbf{Acknowledgments}: The authors would like to thank the referee for his/her careful reading of the text and for helpful comments and suggestions which helped simplify the text. 

\textbf{Conflict of Interest Statement}: On behalf of all authors, the corresponding author states that there is no conflict of interest.

\begin{minipage}{3 in}
	Gili Golan\\
	Department of Mathematics,\\
	Ben Gurion University of the Negev,\\ 
	golangi@bgu.ac.il
\end{minipage}
\begin{minipage}{3 in}
	Mark Sapir\\
	Department of Mathematics,\\
	Vanderbilt University,\\
	m.sapir@vanderbilt.edu
\end{minipage}


\begin{thebibliography}{999999}
	




%
%
%
%

\bibitem{BriS} M. Brin and C. Squier, 
\it Groups of piecewise linear homeomorphisms of the real line, 
\rm Inventiones mathematicae 79 (1985) 485-498.

\bibitem{B} M. Brin, 
\it The free group of rank 2 is a limit of Thompson’s group F, 
\rm Groups, Geometry and Dynamics. Volume 4, Issue 3, 2010, pp. 433-454. 

\bibitem{BG} K. Brown, and R. Geoghegan, 
\it An infinite-dimensional torsion-free $FP_\infty$ group,
\rm Inventiones mathematicae 77, pp. 367–381 (1984).


\bibitem{CFP} J. Cannon, W. Floyd, and W. Parry,
\it  Introductory notes on Richard Thompson's groups.
\rm L'Enseignement Mathematique, 42 (1996), 215--256.

\bibitem{CERT} S. Cleary, M. Elder, A. Rechnitzer, J. Taback,
\it Random subgroups of Thompson's group $F$,
\rm Groups Geom. Dyn. 4 (2010), no. 1, 91-126.
%










\bibitem{GGJ} T. Gelander, G. Golan and K. Juschenko,
\it Invariable generation of Thompson groups,
\rm Journal of Algebra 478 (2017), 261--270.



\bibitem{GS} G. Golan and M. Sapir,
\it On subgroups of R. Thompson group $F$,
\rm Trans. Amer. Math. Soc. 369 (2017), 8857--8878.


\bibitem{G16} G. Golan,
\it The generation problem in Thompson group $F$,
\rm arxiv:1608.02572, to appear in Memoirs of the AMS.

\bibitem{G-RG} G. Golan Polak, 
\it Random Generation of Thompson's group $F$, 
\rm Journal of Algebra, 593 (2022), 507-524.



\bibitem{G22} G. Golan Polak,
\it On Maximal subgroups of Thompson's group $F$,
\rm arxiv:2209.03244. 

\bibitem{G-32} G. Golan Polak,
\it Thompson's group $F$ is almost $\frac{3}{2}$-generated,
\rm arxiv:2210.03564. 


\bibitem{IO} S. V. Ivanov and A. Yu. OlShanskii,
\it Hyperbolic groups and their quotients of bounded exponents,
\rm Transactions of the American Mathematical Society Vol. 348, No. 6 (1996), 2091-2138.

\bibitem{L} Y. Lodha, 
\it Approximating nonabelian free groups by groups of homeomorphisms of the real line,
\rm Journal of Algebra (563), 2020, Pages 292-302.

\bibitem{O} A. Olshanskii,
\bfseries{Geometry of defining relations in groups},
\rm (Nauka, Moscow, 1989) (in Russian); (English translation by Kluwer Publications 1991).

\bibitem{OO} A. Olshanskii and D. Osin, 
\it Large groups and their periodic quotients, 
\rm Proc. of the AMS,  136 (3), 2008, 753–759.

\bibitem{W} J. Wilson, 
\it Finite presentations of pro-p groups and discrete groups,
\rm Invent. Math. 105 (1991), no. 1, 177-183.

%
%


%












%





\end{thebibliography}
\end{document}